\documentclass[11pt]{amsart}
\usepackage{amssymb}
\usepackage{amsfonts}
\usepackage{amsmath}

\newtheorem{theorem}{Theorem}[section]
\newtheorem{proposition}{Proposition}[section]
\newtheorem{lemma}[theorem]{Lemma}
\theoremstyle{definition}

\theoremstyle{remark}

\newtheorem{acknowledgement}{Acknowledgement}
\theoremstyle{corollary}
\newtheorem{corollary}[theorem]{Corollary}
\numberwithin{equation}{section}


\begin{document}
\title{Singular quasilinear elliptic systems with gradient dependence}
\author[H. Dellouche]{Halima Dellouche}
\address{H. Dellouche\\
Applied Mathematics Laboratory (LMA), Faculty of Exact Sciences, A. Mira
Bejaia University, Targa Ouzemour, 06000 Bejaia, Algeria}
\author[A. Moussaoui]{Abdelkrim Moussaoui$^{\ast }$}
\address{A. Moussaoui\\
Applied Mathematics Laboratory (LMA), Faculty of Exact Sciences\\
and Biology departement, Faculty of Natural \& Life Sciences\\
A. Mira Bejaia University, Targa Ouzemour, 06000 Bejaia, Algeria}
\thanks{$\ast $ Corresponding author}

\begin{abstract}
In this paper, we prove existence and regularity of positive solutions for
singular quasilinear elliptic systems involving gradient terms. Our approach
is based on comparison properties, a priori estimates and Schauder's fixed
point theorem.
\end{abstract}

\maketitle

\section{Introduction}

We deal with the following quasilinear elliptic system 
\begin{equation*}
(\mathrm{P})\qquad \left\{ 
\begin{array}{l}
-\Delta _{p_{1}}u=f_{1}(x,u,v,\nabla u,\nabla v)\text{ in }\Omega \\ 
-\Delta _{p_{2}}v=f_{2}(x,u,v,\nabla u,\nabla v)\text{ in }\Omega \\ 
u,v>0\text{ \ \ \ in }\Omega \\ 
u,v=0\text{ \ \ \ on }\partial \Omega%
\end{array}%
\right.
\end{equation*}%
where $\Omega \subset 
\mathbb{R}
^{N}$ $\left( N\geq 2\right) $ is a bounded domain with smooth boundary $%
\partial \Omega $ and $\Delta _{p_{i}}$ stands for the $p_{i}$-Laplacian
differential operator on $W_{0}^{1,p_{i}}(\Omega )$ with $1<p_{i}\leq N,$ $%
i=1,2$. A solution of $(\mathrm{P})$ is understood in the weak sense, that
is a pair $(u,v)\in W_{0}^{1,p_{1}}(\Omega )\times W_{0}^{1,p_{2}}(\Omega )$
with $u,v$ positive $a.e.$ in $\Omega ,$\ and satisfying 
\begin{equation}
\begin{array}{l}
\int_{\Omega }|\nabla u|^{p_{1}-2}\nabla u\nabla \varphi \,dx=\int_{\Omega
}f_{1}(x,u,v,\nabla u,\nabla v)\varphi \,dx, \\ 
\int_{\Omega }|\nabla v|^{p_{2}-2}\nabla v\nabla \psi \,dx=\int_{\Omega
}f_{2}(x,u,v,\nabla u,\nabla v)\psi \,dx,%
\end{array}
\label{10}
\end{equation}%
for all $(\varphi ,\psi )\in W_{0}^{1,p_{1}}(\Omega )\times
W_{0}^{1,p_{2}}(\Omega )$. The nonlinear terms $f_{i}:\Omega \times
(0,+\infty )^{2}\times 
\mathbb{R}
^{2N}\rightarrow (0,+\infty ),$ $i=1,2$, are Carath\'{e}odory functions,
that is, $f_{i}(\cdot ,s_{1},s_{2},\xi _{1},\xi _{2})$ is measurable for
every $(s_{1},s_{2},\xi _{1},\xi _{2})\in (0,+\infty )^{2}\times \mathbb{R}%
^{2N}$ and $f_{i}(x,\cdot ,\cdot ,\cdot ,\cdot )$ is continuous for a.e. $%
x\in \Omega .$

Observe that the dependence on the gradient of the nonlinearities $f_{1}$
and $f_{2}$ deprives system $(\mathrm{P})$ of a variational structure
thereby making it impossible applying variational methods. Moreover, due to
such convection terms, even the application of the so-called topological
methods as sub-super solutions and fixed points technics is not standard.
Another important aspect of problem $(\mathrm{P})$ is that the
reaction-convection term $f_{i}(x,u,v,\nabla u,\nabla v),$ which is
expressed with the solution and its gradient, can exhibit singularities when
the variables $u$ and $v$ approach zero. This occur under the following
growth assumption on functions $f_{1}$ and $f_{2}$, which will be referred
throughout the paper as $(\mathrm{H}_{f})$:

\begin{description}
\item[$(\mathrm{H}_{f})$] There exist constants $M_{i},m_{i}>0,$ $\gamma
_{i},\theta _{i}\geq 0,$ $r_{i}>N$ and $\alpha _{i},\beta _{i}\in 
\mathbb{R}
$ such that 
\begin{equation*}
\begin{array}{c}
m_{i}s_{1}^{\alpha _{i}}s_{2}^{\beta _{i}}\leq f_{i}(x,s_{1},s_{2},\xi
_{1},\xi _{2})\leq M_{i}s_{1}^{\alpha _{i}}s_{2}^{\beta _{i}}+|\xi
_{1}|^{\gamma _{i}}+|\xi _{2}|^{\theta _{i}}\ 
\end{array}%
\end{equation*}
for a.e. $x\in \Omega $, for all $s_{1},s_{2}>0$, and all $\xi _{1},\xi
_{2}\in \mathbb{R}^{N}$, with%
\begin{equation}
-1/r_{i}\leq \alpha _{i}+\beta _{i}<\frac{p_{i}-1}{r_{i}}\text{ \ and \ }%
\max \{\gamma _{i},\theta _{i}\}<\frac{p_{i}-1}{r_{i}},\text{ for }i=1,2.
\label{cdt}
\end{equation}
\end{description}

\mathstrut

Quasilinear convective problem $(\mathrm{P})$ is involved in various
nonlinear processes that occur in many engineering and natural systems. In
biology and physiology, it arises in heat transfer of gas and liquid flow in
plants and animals as well as in cell reactors, incubators, and biomass
systems. In chemical processes, it appears in catalytic and noncatalytic
reactions in exothermic and endothermic reacting flows and in polymer
processing. In geology, it is involved in thermoconvective motion of magmas
and in melt solidification in magma chambers and during volcanic eruptions.
It also arises in global climate energy balance models by coupling the
equation for the mean surface temperature on $\partial \Omega $ with the
equation in $\Omega $ for the ocean temperature. For more inquiries on
modeling physical phenomena we refer to \cite{Ka, DT} and the references
therein. Furthermore, many important singular systems could be incorporated
in the statement of $(\mathrm{P})$. In this respect, we mention for instance
Gierer-Meinhardt system which models some biochemical processes \cite{K}, as
well as singular Lane-Emden system which arises in astrophysics \cite{G}.

Nevertheless, despite its obvious importance, quasilinear singular
convective system $(\mathrm{P})$ has been rarely investigated in the
literature. Actually, in \cite{CLM}, problem $(\mathrm{P})$ subjected to
growth condition like $(\mathrm{H}_{f})$ with 
\begin{equation}
\alpha _{1},\beta _{2}<0<\alpha _{2},\beta _{1},  \label{1}
\end{equation}%
was examined under cooperative structure of nonlinearities $%
f_{1}(x,u,v,\nabla u,\nabla v)$ and $f_{2}(x,u,v,\nabla u,\nabla v)$. This
means that the latter are increasing with respect to $v$ and $u$,
respectively. The complementary situation with respect to (\ref{1}) is the
so-called competitive system. In this context, singular system $(\mathrm{P})$
is studied in \cite{MMZ}, under some appropriate growth conditions that does
not fit our setting in $(\mathrm{H}_{f})$, where the singularities come out
through the solution and its gradient. It is worth pointing out in the
previous papers that a control on the gradient of solutions, besides the one
on solutions, is required hence the use of \cite[Lemma 1]{BE}. At this
point, in \cite{MMZ}, the gradient estimate is obtained through some
specific growth conditions combined with properties of the eigenfunction
corresponding to the first eigenvalue of the operator $-\Delta _{p_{i}}$,
while in \cite{CLM}, besides (\ref{1}), it is requested that $\alpha
_{i}+\beta _{i}\geq 0,$ for $i=1,2$, which is the key ingredient so that 
\cite[Lemma 1]{BE} could be applicable.

Unlike quasilinear systems, meaningful contributions are already available
for semilinear singular convective systems (i.e., $p_{1}=p_{2}=2$). Here, we
quote for instance \cite{ACF, AM} and the references therein, where the
linearity of the principal part has been crucial. The particular case in
singular system $(\mathrm{P})$ when the convection terms $\nabla u$ and $%
\nabla v$ are removed has received a special attention. Relevant
contributions regarding this topic can be found in \cite{AC, DM1, DM2, EPS,
KM, MM3, MM2, MM1} and the references given there.

Motivated by the aforementioned papers, the aim of this work is to establish
existence and regularity of (positive) solutions for quasilinear singular
convective system $(\mathrm{P})$ subjected to the growth condition $(\mathrm{%
H}_{f})$. The main result is formulated as follows.

\begin{theorem}
\label{T1}Under assumption $(\mathrm{H}_{f}),$ system $(\mathrm{P})$ admits
a positive solution $(u,v)$ in $C_{0}^{1,\sigma }(\overline{\Omega })\times
C_{0}^{1,\sigma }(\overline{\Omega })$ for certain $\sigma \in (0,1)$.
\end{theorem}

The main technical difficulty consists in the involvement of gradient terms
in system $(\mathrm{P})$ as well as the presence of singular terms that can
occur under hypothesis $(\mathrm{H}_{f})$. This difficulty is heightened by
the very marked singularity character of $(\mathrm{P})$ that arises from $(%
\mathrm{H}_{f})$ when $\alpha _{i}+\beta _{i}<0$ for $i=1,2$. Moreover,
unlike in the earlier papers, neither cooperative nor competitive structure
on the system is imposed. In fact, these both complementary structures for
the system $(\mathrm{P})$ are handled simultaneously without referring to
them despite their important structural disparity that makes functions $%
f_{1} $ and $f_{2}$ behaving in a drastically different way.

Our approach is chiefly based on Schauder's fixed point Theorem (see \cite{Z}%
). We establish some comparison arguments through suitable functions with an
adjustment of adequate constants. These functions beget a positive rectangle
providing a localization of an eventual solution. On the other hand, we
provide a priori estimate on the gradient of solutions by exploiting the
growth condition $(\mathrm{H}_{f})$ together with \cite[Theorem 3.1; Remark
3.3]{CM}. Thence, this established a control on solutions and their gradient
that combined together leads to a fixed point through Schauder's Theorem,
which is actually a solution of $(\mathrm{P})$. It is worth pointing out
that our argument is substantially different from that developed in \cite%
{CLM} where extremal solutions concept is relevant. Another significant
feature of our existence result concerns the regularity part. At this point,
Theorem \ref{T4} stated in section \ref{S2} provides a regularity result for
problem $(\mathrm{P})$ extending \cite[Lemma 3.1]{Hai} to singular systems
involving convection terms.

The rest of the paper is organized as follows. Section \ref{S2} establishes
gradient estimates and regularity result; Section \ref{S3} deals with
comparison properties; Section \ref{4} contains the proof of the main result.

\section{Gradient estimate and regularity}

\label{S2}

In the sequel, the Banach spaces $W^{1,p}(\Omega )$ and $L^{p}(\Omega )$ are
equipped with the usual norms $\Vert \cdot \Vert _{1,p}$ and $\Vert \cdot
\Vert _{p}$, respectively, whereas the space $W_{0}^{1,p}(\Omega )$ is
endowed with the equivalent norm $\Vert u\Vert =(\int_{\Omega }|\nabla
u|^{p}\,\mathrm{d}x)^{\frac{1}{p}}$. We denote by $p^{\prime }=\frac{p}{p-1}$%
. We also utilize $C^{1}(\overline{\Omega })$ and $C_{0}^{1,\sigma }(%
\overline{\Omega })$ such that for $\sigma \in (0,1)$ we have $%
C_{0}^{1,\sigma }(\overline{\Omega })=\{u\in C^{1,\sigma }(\overline{\Omega }%
):u=0$ on $\partial \Omega \}$. The continuity of the embedding $%
W_{0}^{1,p}(\Omega )$ in $L^{r}(\Omega ),$ for $1\leq r\leq \frac{Np}{N-p},$
guarantees the existence of a constant $c_{r}>0$ such that 
\begin{equation}
\Vert u\Vert _{r}\leq c_{r}\Vert u\Vert \ \text{ for all}\ u\in
W_{0}^{1,p}(\Omega ).  \label{embed}
\end{equation}

Let us introduce the quasilinear system%
\begin{equation*}
(\mathrm{P}_{h})\qquad \left\{ 
\begin{array}{l}
-\Delta _{p_{1}}u=h_{1}(x,u,v,\nabla u,\nabla v)\text{ in }\Omega \\ 
-\Delta _{p_{2}}v=h_{2}(x,u,v,\nabla u,\nabla v)\text{ in }\Omega \\ 
u,v=0\text{ \ \ \ on }\partial \Omega ,%
\end{array}%
\right.
\end{equation*}%
where $h_{i}:\Omega \times 
\mathbb{R}
^{2}\times 
\mathbb{R}
^{2N}\rightarrow 
\mathbb{R}
$ is a Carath\'{e}odory function satisfying the following assumption:

\begin{description}
\item[$(\mathrm{H}_{h})$] There exist constants $M>0,$ $r_{i}>N$ and $\mu
_{i}<0\leq \gamma _{i},\theta _{i}$ such that%
\begin{equation*}
\max_{i=1,2}|h_{i}(x,s_{1},s_{2},\xi _{1},\xi _{2})|\leq M(d(x)^{\mu
_{i}}+\left\vert \xi _{1}\right\vert ^{\gamma _{i}}+\left\vert \xi
_{2}\right\vert ^{\theta _{i}}),
\end{equation*}%
for a.e. $x\in \Omega $, for all $s_{i}\in 
\mathbb{R}
$, and all $\xi _{i}\in \mathbb{R}^{N}$ such that%
\begin{equation}
\mu _{i}>-\frac{1}{r_{i}}\text{ \ and \ }\max \{\gamma _{i},\theta _{i}\}<%
\frac{p_{i}-1}{r_{i}},\text{ for }i=1,2.  \label{5}
\end{equation}
\end{description}

The following result is a key point in the proof of Theorem \ref{T1}.

\begin{theorem}
\label{T4}Assume $(\mathrm{H}_{h})$ holds true. Then, there is a constant $%
k_{p}>0,$ depending only on $p_{1},p_{2}$ and $\Omega ,$ such that%
\begin{equation*}
\begin{array}{l}
\Vert \nabla u\Vert _{\infty }\leq k_{p}||h_{1}(\cdot ,u,v,\nabla u,\nabla
v)\Vert _{r_{1}}^{\frac{1}{p_{1}-1}}%
\end{array}%
\end{equation*}%
and%
\begin{equation*}
\begin{array}{l}
\Vert \nabla v\Vert _{\infty }\leq k_{p}||h_{2}(\cdot ,u,v,\nabla u,\nabla
v)\Vert _{r_{2}}^{\frac{1}{p_{2}-1}}.%
\end{array}%
\end{equation*}%
Moreover, there are constants $R>0$ and $\sigma \in (0,1)$ such that all
solutions $(u,v)\in W_{0}^{1,p_{1}}(\Omega )\times W_{0}^{1,p_{2}}(\Omega )$
of problem $(\mathrm{P}_{h})$ belong to $C_{0}^{1,\sigma }(\overline{\Omega }%
)\times C_{0}^{1,\sigma }(\overline{\Omega })$ and satisfy the estimate%
\begin{equation}
\left\Vert u\right\Vert _{C^{1,\sigma }(\overline{\Omega })},\left\Vert
v\right\Vert _{C^{1,\sigma }(\overline{\Omega })}<R\text{.}  \label{regu}
\end{equation}
\end{theorem}

\begin{proof}
Let $(u,v)\in W_{0}^{1,p_{1}}(\Omega )\times W_{0}^{1,p_{2}}(\Omega )$ be a
solution of $(\mathrm{P}_{h})$. Multiplying the first equation in $(\mathrm{P%
}_{h})$ by $u,$\ integrating over $\Omega $ and using $(\mathrm{H}_{h})$,
one gets%
\begin{equation}
\begin{array}{l}
\int_{\Omega }|\nabla u|^{p_{1}}\text{ }\mathrm{d}x=\int_{\Omega
}h_{1}(x,u,v,\nabla u,\nabla v)u\text{ }\mathrm{d}x \\ 
\leq M\int_{\Omega }(d(x)^{\mu _{1}}+\left\vert \nabla u\right\vert ^{\gamma
_{1}}+\left\vert \nabla v\right\vert ^{\theta _{1}})u\text{ }\mathrm{d}x.%
\end{array}
\label{43}
\end{equation}%
Since $\mu _{1}>-1$, by Hardy-Sobolev inequality (see, e.g., \cite[Lemma 2.3]%
{AC}), there exists a positive constant $C_{1}$ such that%
\begin{equation}
\int_{\Omega }d(x)^{\mu _{1}}u\text{ }\mathrm{d}x\leq C_{1}\Vert \nabla
u\Vert _{p_{1}}\text{.}  \label{45}
\end{equation}%
Using (\ref{5}), Young inequality and (\ref{embed}) imply that%
\begin{equation}
\begin{array}{l}
\int_{\Omega }\left\vert \nabla u\right\vert ^{\gamma _{1}}u\text{ }\mathrm{d%
}x\leq \varepsilon \int_{\Omega }|u|^{p_{1}}\text{ }\mathrm{d}%
x+c_{\varepsilon }\int_{\Omega }|\nabla u|^{\gamma _{1}p_{1}^{\prime }}\text{
}\mathrm{d}x \\ 
\leq \varepsilon \left\Vert \nabla u\right\Vert
_{p_{1}}^{p_{1}}+c_{\varepsilon }\left\Vert \nabla u\right\Vert
_{p_{1}}^{\gamma _{1}p_{1}^{\prime }},%
\end{array}%
\end{equation}%
\begin{equation}
\begin{array}{l}
\int_{\Omega }\left\vert \nabla v\right\vert ^{\theta _{1}}u\text{ }\mathrm{d%
}x\leq \varepsilon \int_{\Omega }|u|^{p_{1}}\text{ }\mathrm{d}%
x+c_{\varepsilon }\int_{\Omega }|\nabla v|^{\theta _{1}p_{1}^{\prime }}\text{
}\mathrm{d}x \\ 
\leq \varepsilon \left\Vert \nabla u\right\Vert
_{p_{1}}^{p_{1}}+c_{\varepsilon }\left\Vert \nabla v\right\Vert
_{p_{2}}^{\theta _{1}p_{1}^{\prime }},%
\end{array}%
\end{equation}%
for every $\varepsilon >0$ and with a certain constant $c_{\varepsilon }>0$
depending on $\varepsilon $.

Similarly, multiplying the second equation in $(\mathrm{P}_{h})$ by $v,$\
integrating over $\Omega $ and repeating the previous argument leads to%
\begin{equation}
\begin{array}{l}
\int_{\Omega }|\nabla v|^{p_{2}}\text{ }\mathrm{d}x=\int_{\Omega
}h_{2}(u,v,\nabla u,\nabla v)u\text{ }\mathrm{d}x \\ 
\leq M\int_{\Omega }(d(x)^{\mu _{2}}+\left\vert \nabla u\right\vert ^{\gamma
_{2}}+\left\vert \nabla v\right\vert ^{\theta _{2}})v\text{ }\mathrm{d}x%
\end{array}%
\end{equation}%
with%
\begin{equation}
\int_{\Omega }d(x)^{\mu _{2}}v\text{ }\mathrm{d}x\leq C_{2}\Vert \nabla
v\Vert _{p_{2}},
\end{equation}%
\begin{equation}
\int_{\Omega }\left\vert \nabla v\right\vert ^{\theta _{2}}v\text{ }\mathrm{d%
}x\leq \varepsilon \left\Vert \nabla v\right\Vert
_{p_{2}}^{p_{2}}+c_{\varepsilon }^{\prime }\left\Vert \nabla v\right\Vert
_{p_{2}}^{\theta _{2}p_{2}^{\prime }},
\end{equation}%
\begin{equation}
\int_{\Omega }\left\vert \nabla u\right\vert ^{\gamma _{2}}v\text{ }\mathrm{d%
}x\leq \varepsilon \left\Vert \nabla v\right\Vert
_{p_{2}}^{p_{2}}+c_{\varepsilon }^{\prime }\left\Vert \nabla u\right\Vert
_{p_{1}}^{\gamma _{2}p_{2}^{\prime }}.  \label{44}
\end{equation}%
for $\varepsilon >0$ and a constant $c_{\varepsilon }^{\prime }>0$. On the
other hand, since $r_{i}>N>\frac{p_{i}}{p_{j}},$ for $i,j=1,2,i\neq j$, it
follows, from (\ref{5}), that%
\begin{equation}
\max_{i=1,2}(\gamma _{i}p_{i}^{\prime })<p_{1}\text{ \ and \ }%
\max_{i=1,2}(\theta _{i}p_{i}^{\prime })<p_{2}.  \label{47}
\end{equation}

Then, gathering (\ref{43})-(\ref{44}) together, on account of (\ref{47}) and
by taking $\varepsilon $ small, there is a constant $c_{0}>0,$ independent
of the solution $(u,v)$, such that 
\begin{equation}
\left\Vert \nabla u\right\Vert _{p_{1}},\left\Vert \nabla v\right\Vert
_{p_{2}}\leq c_{0}.  \label{46}
\end{equation}%
Combining the last estimate with (\ref{45}), through the growth condition $(%
\mathrm{H}_{h})$, it follows that 
\begin{equation}
\Vert h_{i}(u,v,\nabla u,\nabla v)\Vert _{r_{i}}\leq c_{1},  \label{2'}
\end{equation}%
with a constant $c_{1}>0$ independent of the solution $(u,v)$. At this point
we make use of the assumption that $r_{i}>N$, which enables us to refer to
the gradient bound in \cite[Theorem 3.1; Remark 3.3]{CM}. Consequently, due
to (\ref{2'}), we infer that 
\begin{equation}
\Vert \nabla u\Vert _{\infty }\leq k_{p}||h_{1}(\cdot ,u,v,\nabla u,\nabla
v)\Vert _{r_{1}}^{\frac{1}{p_{1}-1}}  \label{3}
\end{equation}%
and 
\begin{equation}
\Vert \nabla v\Vert _{\infty }\leq k_{p}||h_{2}(\cdot ,u,v,\nabla u,\nabla
v)\Vert _{r_{2}}^{\frac{1}{p_{2}-1}},  \label{3*}
\end{equation}%
where the constant $k_{p}>0$ depends only on $p_{1},p_{2}$ and $\Omega $. On
the other hand, by $(\mathrm{H}_{h})$ and (\ref{2'})-(\ref{3*}) we have 
\begin{equation*}
\begin{array}{l}
|h_{i}(x,u,v,\nabla u,\nabla v)|\leq M\left( d(x)^{\mu _{i}}+\left\vert
\nabla u\right\vert ^{\gamma _{i}}+\left\vert \nabla v\right\vert ^{\theta
_{i}}\right) \\ 
\leq M\left( d(x)^{\mu _{i}}+\left\Vert \nabla u\right\Vert _{\infty
}^{\gamma _{i}}+\left\Vert \nabla v\right\Vert _{\infty }^{\theta
_{i}}\right) \\ 
\leq M\left( d(x)^{\mu _{i}}+(k_{p}c_{1}^{\frac{1}{p_{1}-1}})^{\gamma
_{i}}+(k_{p}c_{1}^{\frac{1}{p_{2}-1}})^{\theta _{i}}\right) \\ 
\leq \hat{M}\left( d(x)^{\mu _{i}}+1\right) \text{ in }\Omega ,%
\end{array}%
\end{equation*}%
for some constant $\hat{M}>0$ independent of $u$ and $v$. Thus, since $\mu
_{i}>-1$, the nonlinear regularity up to the boundary in \cite[Lemma 3.1]%
{Hai} applies, showing that the solution $(u,v)$ is bounded in $%
C_{0}^{1,\sigma }(\overline{\Omega })\times C_{0}^{1,\sigma }(\overline{%
\Omega })$ for certain $\sigma \in (0,1)$. This ends the proof.
\end{proof}

In the scalar case of $(\mathrm{P}_{h})$, the previous regularity Theorem %
\ref{T4} establishes an extension of \cite[Lemma 3.1]{Hai} to singular
problems involving convection terms. It is formulated as follows.

\begin{corollary}
For $1<p\leq N$ and $M>0,$ let $h:\Omega \times 
\mathbb{R}
\times 
\mathbb{R}
^{N}\rightarrow 
\mathbb{R}
$ be a Carath\'{e}odory function satisfying%
\begin{equation*}
|h(x,s,\xi )|\leq M(d(x)^{\mu }+\left\vert \xi \right\vert ^{\gamma }),
\end{equation*}%
for a.e. $x\in \Omega $, for all $s\in 
\mathbb{R}
,$ and all $\xi \in \mathbb{R}^{N}$ with%
\begin{equation*}
r>N\text{ and }-\frac{1}{r}<\mu <0\leq \gamma <\frac{p-1}{r}.
\end{equation*}%
Then, there are constants $R>0$ and $\sigma \in (0,1)$ such that all
solutions $u\in W_{0}^{1,p}(\Omega )$ of Dirichlet problem%
\begin{equation*}
\left\{ 
\begin{array}{l}
-\Delta _{p}u=h(x,u,\nabla u)\text{ in }\Omega \\ 
u=0\text{ \ \ \ on }\partial \Omega ,%
\end{array}%
\right.
\end{equation*}%
belong to $C_{0}^{1,\sigma }(\overline{\Omega })\times C_{0}^{1,\sigma }(%
\overline{\Omega })$ and satisfy the estimate $\left\Vert u\right\Vert
_{C^{1,\sigma }(\overline{\Omega })}<R$. Moreover, there is a constant $%
k_{p}>0,$ depending only on $p$ and $\Omega ,$ such that 
\begin{equation*}
\Vert \nabla u\Vert _{\infty }\leq k_{p}||h(\cdot ,u,\nabla u)\Vert _{r}^{%
\frac{1}{p-1}}.
\end{equation*}
\end{corollary}

\section{Comparison properties}

\label{S3}

Let $y_{i},z_{i}\in C^{1}(\overline{\Omega })$ satisfy%
\begin{equation}
-\Delta _{p_{i}}y_{i}(x)=1+d(x)^{\alpha _{i}+\beta _{i}}\text{ \ in }\Omega ,%
\text{ }y_{i}(x)=0\text{ \ on }\partial \Omega ,  \label{7}
\end{equation}%
and%
\begin{equation}
-\Delta _{p_{i}}z_{i}(x)=\left\{ 
\begin{array}{ll}
d(x)^{\alpha _{i}+\beta _{i}} & \text{in }\Omega \backslash \overline{\Omega 
}_{\delta } \\ 
-1 & \text{in }\Omega _{\delta }%
\end{array}%
\right. ,\text{ }z_{i}(x)=0\text{ \ on }\partial \Omega  \label{9}
\end{equation}%
where $\Omega _{\delta }:=\{x\in \Omega :d(x)<\delta \}$, with a fixed $%
\delta >0$ sufficiently small, and the exponents $\alpha _{i},\beta _{i}$
satisfy 
\begin{equation}
-1<\alpha _{i}+\beta _{i}<p_{i}-1,\text{ for }i=1,2.  \label{11}
\end{equation}%
Hardy-Sobolev inequality (see, e.g., \cite[Lemma 2.3]{AC}) guarantees that
the right-hand sides of (\ref{7}) and (\ref{9}) belong to $%
W^{-1,p_{i}^{\prime }}(\Omega )$. Consequently, Minty-Browder Theorem (see 
\cite[Theorem V.15]{Brezis}) ensures the existence of unique $y_{i}$ and $%
z_{i}$ in (\ref{9}) and (\ref{7}).

\begin{lemma}
\label{L2}There are positive constants $c_{0}$ and $c_{1}$ such that%
\begin{equation}
c_{0}d(x)\leq z_{i}(x)\leq y_{i}(x)\leq c_{1}d(x)\text{ \ for all }x\in
\Omega ,i=1,2.  \label{4}
\end{equation}
\end{lemma}

\begin{proof}
By (\ref{7}) and (\ref{9}), it is readily seen that $z_{i}(x)\leq y_{i}(x)$
for all $x\in \Omega ,$ for $i=1,2.$ Moreover, the Strong Maximum Principle
together with \cite[Corollary 3.1]{Hai} entail $c_{0}d(x)\leq z_{i}(x)$ in $%
\Omega ,$ for $\delta >0$ sufficiently small in (\ref{9}). Thus, the proof
is completed by showing the last inequality in (\ref{4}). To this end, let
consider the unique solution $\hat{w}_{i}\in C_{0}^{1,\tau }(\overline{%
\Omega })$ homogeneous Dirichlet problem 
\begin{equation}
-\Delta _{p_{i}}\hat{w}_{i}=1\text{ in }\Omega ,\text{ }\hat{w}_{i}=0\text{
on }\partial \Omega   \label{5bis}
\end{equation}%
that satisfies 
\begin{equation}
\Vert \nabla \hat{w}_{i}\Vert _{\infty }\leq \hat{L},\text{ for certain
constant }\hat{L}>0,\text{ }i=1,2\text{.}  \label{56}
\end{equation}%
Since $\partial \Omega $ is smooth, we can find $\delta >0$ and $\Pi \in
C^{1}(\Omega _{\delta },\partial \Omega )$ fulfilling 
\begin{equation}
d(x)=|x-\Pi (x)|,\;\;\frac{x-\Pi (x)}{|x-\Pi (x)|}=-\eta (\Pi (x)),\;\;]\Pi
(x),x]\subseteq \Omega ,\;\;x\in \Omega _{\delta },  \label{propPi}
\end{equation}%
where $\eta (x)$ indicates the outward unit normal vector to $\partial
\Omega $ at its point $x$. Thus, the Mean Value Theorem, when combined with (%
\ref{56}), lead to 
\begin{equation}
|\hat{w}_{i}(x)|=|\hat{w}_{i}(x)-\hat{w}_{i}(\Pi (x))|\leq \hat{L}|x-\Pi
(x)|=\hat{L}d(x)\quad \forall \,x\in \Omega _{\delta }\,.  \label{MVT1}
\end{equation}%
Define 
\begin{equation*}
L:=\max \left\{ \hat{L},\,\max_{\Omega \setminus \Omega _{\delta }}\frac{%
\hat{w}_{i}}{d}\,,\,i=1,2\right\} .
\end{equation*}%
On account of (\ref{MVT1}), one evidently has 
\begin{equation}
\hat{w}_{i}\leq Ld(x)\text{ in }\Omega .  \label{58}
\end{equation}%
Let us ssume that $\alpha _{i}+\beta _{i}\geq 0$ in (\ref{11}). Then, there
is a constant $\hat{c}>0$ such that $1+d(x)^{\alpha _{i}+\beta _{i}}\leq 
\hat{c}$. From (\ref{5bis}) and (\ref{7}) one derives that%
\begin{equation*}
-\Delta _{p_{i}}(\hat{c}^{\frac{-1}{p_{i}-1}}y_{i})\leq -\Delta _{p_{i}}\hat{%
w}_{i}\text{ in }\Omega ,
\end{equation*}%
which, by (\ref{58}) and the weak comparison principle (see \cite[Lemma 3.1]%
{T}), implies that $y_{i}(x)\leq c_{1}d(x)$ in $\Omega ,$ for certain
constant $c_{1}$. Now, assume that $\alpha _{i}+\beta _{i}<0$ in (\ref{11})
and let $w_{i}\in C^{1}(\overline{\Omega }),$ be the unique weak solution of
Dirichlet problem 
\begin{equation}
\left\{ 
\begin{array}{ll}
-\Delta _{p_{i}}w_{i}=w_{i}^{-\gamma } & \text{in }\Omega , \\ 
w_{i}>0 & \text{in }\Omega , \\ 
w_{i}=0 & \text{on }\partial \Omega 
\end{array}%
\right. ,i=1,2,  \label{20}
\end{equation}%
where 
\begin{equation}
\max_{i=1,2}|\alpha _{i}+\beta _{i}|<\gamma <1.  \label{65}
\end{equation}%
It is well known that%
\begin{equation}
c_{2}d(x)\leq w_{i}(x)\leq c_{3}d(x),\text{ }i=1,2  \label{6}
\end{equation}%
with positive constants $c_{2},c_{3}$ (see \cite{GST}). Due to (\ref{65}),
one can find a constant $L>0$ such that%
\begin{equation}
\max_{i=1,2}(1+d(x)^{\alpha _{i}+\beta _{i}})d(x)^{\gamma }\leq L\text{ \ in 
}\overline{\Omega }.  \label{8}
\end{equation}

On account of (\ref{7}), (\ref{8}) and (\ref{6}), it follows that%
\begin{equation*}
\begin{array}{l}
-\Delta _{p_{i}}y_{i}(x)=1+d(x)^{\alpha _{i}+\beta _{i}}\leq Ld(x)^{-\gamma }
\\ 
\leq L(c_{3}w_{i})^{-\gamma }=-\Delta _{p_{i}}((Lc_{3}^{-\gamma })^{\frac{1}{%
p_{i}-1}}w_{i}(x))\text{ \ in }\Omega \text{.}%
\end{array}%
\end{equation*}%
Then, the weak comparison principle (see \cite[Lemma 3.1]{T}) leads to%
\begin{equation*}
\begin{array}{c}
y_{i}(x)\leq (Lc_{3}^{-\gamma })^{\frac{1}{p_{i}-1}}w_{i}(x)\leq L^{\frac{1}{%
p_{i}-1}}c_{3}^{1-\frac{\gamma }{p_{i}-1}}d(x)\text{ \ in }\Omega .%
\end{array}%
\end{equation*}%
This ends the proof.
\end{proof}

\mathstrut

Set%
\begin{equation}
(\underline{u},\underline{v})=C^{-1}(z_{1},z_{2})\quad \text{and}\quad (%
\overline{u},\overline{v})=C(y_{1},y_{2})  \label{sub-super}
\end{equation}%
where $C>1$ is a constant. Obviously, $\underline{u}\leqslant \underline{v}$
and $\overline{u}\leqslant \overline{v}$ in $\overline{\Omega }.$

The following result allows us to achieve useful comparison properties.

\begin{proposition}
\label{P1}Assume (\ref{11}) holds true and let $\gamma _{i},\theta _{i}\geq
0 $ be real constants such that%
\begin{equation}
\max \{\gamma _{i},\theta _{i}\}<p_{i}-1,\text{ \ for }i=1,2.  \label{c}
\end{equation}%
Then, for $C>0$ sufficiently large in (\ref{sub-super}), it holds%
\begin{equation*}
\begin{array}{l}
-\Delta _{p_{1}}\underline{u}\leq m_{1}w_{1}^{\alpha _{1}}w_{2}^{\beta _{1}}%
\text{\ \ in }\Omega ,%
\end{array}%
\end{equation*}%
\begin{equation*}
-\Delta _{p_{2}}\underline{v}\leq m_{2}w_{1}^{\alpha _{2}}w_{2}^{\beta _{2}}%
\text{ \ in }\Omega ,
\end{equation*}%
\begin{equation*}
\begin{array}{l}
-\Delta _{p_{1}}\overline{u}\geq M_{1}w_{1}^{\alpha _{1}}w_{2}^{\beta
_{1}}+2C^{\max \{\gamma _{1},\theta _{1}\}}\text{\ \ in }\Omega%
\end{array}%
\end{equation*}%
and%
\begin{equation*}
\begin{array}{l}
-\Delta _{p_{2}}\overline{v}\geq M_{2}w_{1}^{\alpha _{2}}w_{2}^{\beta
_{2}}+2C^{\max \{\gamma _{2},\theta _{2}\}}\text{ \ in }\Omega ,%
\end{array}%
\end{equation*}%
whenever $(w_{1},w_{2})\in \lbrack \underline{u},\overline{u}]\times \lbrack 
\underline{v},\overline{v}]$.
\end{proposition}

\begin{proof}
On the basis of (\ref{sub-super}) and Lemma \ref{L2}, for all $%
(w_{1},w_{2})\in \lbrack \underline{u},\overline{u}]\times \lbrack 
\underline{v},\overline{v}],$ on has%
\begin{eqnarray}
w_{1}^{\alpha _{i}}w_{2}^{\beta _{i}} &\geq &\left\{ 
\begin{array}{ll}
\underline{u}^{\alpha _{i}}\underline{v}^{\beta _{i}} & \text{if }\alpha
_{i},\beta _{i}>0 \\ 
\overline{u}^{\alpha _{i}}\underline{v}^{\beta _{i}} & \text{if }\alpha
_{i}<0<\beta _{i} \\ 
\underline{u}^{\alpha _{i}}\overline{v}^{\beta _{i}} & \text{if }\beta
_{i}<0<\alpha _{i} \\ 
\overline{u}^{\alpha _{i}}\overline{v}^{\beta _{i}} & \text{if }\alpha
_{i},\beta _{i}<0%
\end{array}%
\right.  \label{50} \\
&\geq &\left\{ 
\begin{array}{ll}
(C^{-1}c_{0}d(x))^{\alpha _{i}+\beta _{i}} & \text{if }\alpha _{i},\beta
_{i}>0 \\ 
(Cc_{1})^{\alpha _{i}}(C^{-1}c_{0})^{\beta _{i}}d(x)^{\alpha _{i}+\beta _{i}}
& \text{if }\alpha _{i}<0<\beta _{i} \\ 
(C^{-1}c_{0})^{\alpha _{i}}(Cc_{1})^{\beta _{i}}d(x)^{\alpha _{i}+\beta _{i}}
& \text{if }\beta _{i}<0<\alpha _{i} \\ 
(Cc_{1})^{\alpha _{i}+\beta _{i}}d(x)^{\alpha _{i}+\beta _{i}} & \text{if }%
\alpha _{i},\beta _{i}<0%
\end{array}%
\right.  \notag \\
&\geq &\tilde{c}_{0}C^{-(|\alpha _{i}|+|\beta _{i}|)}d(x)^{\alpha _{i}+\beta
_{i}}  \notag
\end{eqnarray}%
and%
\begin{eqnarray}
w_{1}^{\alpha _{i}}w_{2}^{\beta _{i}} &\leq &\left\{ 
\begin{array}{ll}
\overline{u}^{\alpha _{i}}\overline{v}^{\beta _{i}} & \text{if }\alpha
_{i},\beta _{i}>0 \\ 
\underline{u}^{\alpha _{i}}\overline{v}^{\beta _{i}} & \text{if }\alpha
_{i}<0<\beta _{i} \\ 
\overline{u}^{\alpha _{i}}\underline{v}^{\beta _{i}} & \text{if }\beta
_{i}<0<\alpha _{i} \\ 
\underline{u}^{\alpha _{i}}\underline{v}^{\beta _{i}} & \text{if }\alpha
_{i},\beta _{i}<0%
\end{array}%
\right.  \label{12} \\
&\leq &\left\{ 
\begin{array}{ll}
(Cc_{1}d(x))^{\alpha _{i}+\beta _{i}} & \text{if }\alpha _{i},\beta _{i}>0
\\ 
(C^{-1}c_{0})^{\alpha _{i}}(Cc_{1})^{\beta _{i}}d(x)^{\alpha _{i}+\beta _{i}}
& \text{if }\alpha _{i}<0<\beta _{i} \\ 
(Cc_{1})^{\alpha _{i}}(C^{-1}c_{0})^{\beta _{i}}d(x)^{\alpha _{i}+\beta _{i}}
& \text{if }\beta _{i}<0<\alpha _{i} \\ 
(C^{-1}c_{0})^{\alpha _{i}}(C^{-1}c_{0})^{\beta _{i}}d(x)^{\alpha _{i}+\beta
_{i}} & \text{if }\alpha _{i},\beta _{i}<0%
\end{array}%
\right.  \notag \\
&\leq &\tilde{c}_{1}C^{|\alpha _{i}|+|\beta _{i}|}d(x)^{\alpha _{i}+\beta
_{i}},  \notag
\end{eqnarray}%
where $\tilde{c}_{0},\tilde{c}_{1}$ are positive constants depending on $%
c_{0},c_{1},\alpha _{i},\beta _{i}$. From (\ref{sub-super}), (\ref{9}), (\ref%
{7}) and (\ref{c}), one gets 
\begin{equation}
\begin{array}{l}
-\Delta _{p_{1}}\underline{u}=C^{-(p_{1}-1)}\left\{ 
\begin{array}{ll}
d(x)^{\alpha _{1}+\beta _{1}} & \text{in }\Omega \backslash \overline{\Omega 
}_{\delta } \\ 
-1 & \text{in }\Omega _{\delta }%
\end{array}%
\right. \\ 
\leq \tilde{c}_{0}m_{1}C^{-(|\alpha _{1}|+|\beta _{1}|)}d(x)^{\alpha
_{1}+\beta _{1}}\text{ in }\Omega ,%
\end{array}%
\end{equation}%
\begin{equation}
\begin{array}{l}
-\Delta _{p_{2}}\underline{v}=C^{-(p_{2}-1)}\left\{ 
\begin{array}{ll}
d(x)^{\alpha _{2}+\beta _{2}} & \text{in }\Omega \backslash \overline{\Omega 
}_{\delta } \\ 
-1 & \text{in }\Omega _{\delta }%
\end{array}%
\right. \\ 
\leq \tilde{c}_{0}m_{2}C^{-(|\alpha _{2}|+|\beta _{2}|)}d(x)^{\alpha
_{2}+\beta _{2}}\text{ in }\Omega ,%
\end{array}%
\end{equation}%
\begin{equation}
\begin{array}{l}
-\Delta _{p_{1}}\overline{u}=C^{p_{1}-1}\left( 1+d(x)^{\alpha _{1}+\beta
_{1}}\right) \\ 
\geq \tilde{c}_{1}M_{1}C^{|\alpha _{1}|+|\beta _{1}|}d(x)^{\alpha _{1}+\beta
_{1}}+2C^{\max \{\gamma _{1},\theta _{1}\}}\text{ \ in }\Omega%
\end{array}%
\end{equation}%
and%
\begin{equation}
\begin{array}{l}
-\Delta _{p_{2}}\overline{v}=C^{p_{2}-1}(1+d(x)^{\alpha _{2}+\beta _{2}}) \\ 
\geq \tilde{c}_{1}M_{2}C^{|\alpha _{2}|+|\beta _{2}|}d(x)^{\alpha _{2}+\beta
_{2}}+2C^{\max \{\gamma _{2},\theta _{2}\}}\text{ \ in }\Omega ,%
\end{array}
\label{51}
\end{equation}%
provided $C>1$ sufficiently large. Consequently, combining (\ref{50})-(\ref%
{51}) the desired estimates follow.
\end{proof}

\section{Existence of solutions}

\label{S4}

By using the functions in (\ref{sub-super}), let us introduce the$\mathcal{\ 
}$set%
\begin{equation*}
\begin{array}{l}
\mathcal{O}_{C}:=\left\{ 
\begin{array}{c}
(\mathrm{u},\mathrm{v})\in C_{0}^{1}(\overline{\Omega })^{2}:\underline{u}%
\leqslant \mathrm{u}\leqslant \overline{u},\text{ }\underline{v}\leqslant 
\mathrm{v}\leqslant \overline{v}\text{ in }\Omega \\ 
and\text{ }\Vert \nabla \mathrm{u}\Vert _{\infty },\Vert \nabla \mathrm{v}%
\Vert _{\infty }\leqslant C%
\end{array}%
\right\} ,%
\end{array}%
\end{equation*}%
which is closed, bounded and convex in $C_{0}^{1}(\overline{\Omega })\times
C_{0}^{1}(\overline{\Omega }).$ Define the operator $\mathcal{T}:\mathcal{O}%
_{C}\rightarrow C_{0}^{1}(\overline{\Omega })\times C_{0}^{1}(\overline{%
\Omega })$ by $\mathcal{T}(w_{1},w_{2})=(u,v)$ for all $(w_{1},w_{2})\in 
\mathcal{O}_{C}$, where $(u,v)$ is required to satisfy%
\begin{equation*}
(\mathrm{P}_{w})\qquad \left\{ 
\begin{array}{l}
-\Delta _{p_{1}}u=f_{1}(x,w_{1},w_{2},\nabla w_{1},\nabla w_{2})\quad \text{
in }\Omega \\ 
-\Delta _{p_{2}}v=f_{2}(x,w_{1},w_{2},\nabla w_{1},\nabla w_{2})\quad \text{
in }\Omega \\ 
u,v=0\text{ \ \ \ on }\partial \Omega .%
\end{array}%
\right.
\end{equation*}%
In view of the definition $\mathcal{O}_{C}$ and by $(\mathrm{H}_{f})$, we
have%
\begin{equation}
\begin{array}{l}
m_{i}w_{1}^{\alpha _{i}}w_{2}^{\beta _{i}}\leq f_{i}(x,w_{1},w_{2},\nabla
w_{1},\nabla w_{2}) \\ 
\leq M_{i}w_{1}^{\alpha _{i}}w_{2}^{\beta _{i}}+\left\vert \nabla
w_{1}\right\vert ^{\gamma _{i}}+\left\vert \nabla w_{2}\right\vert ^{\theta
_{i}} \\ 
\leq M_{i}w_{1}^{\alpha _{i}}w_{2}^{\beta _{i}}+\left\Vert \nabla
w_{1}\right\Vert _{\infty }^{\gamma _{i}}+\left\Vert \nabla w_{2}\right\Vert
_{\infty }^{\theta _{i}} \\ 
\leq M_{i}w_{1}^{\alpha _{i}}w_{2}^{\beta _{i}}+2C^{\max \{\gamma
_{i},\theta _{i}\}}\text{ in }\Omega ,\text{ for }i=1,2.%
\end{array}
\label{17}
\end{equation}%
Hence, recalling (\ref{12}) we achieve%
\begin{equation}
\begin{array}{l}
f_{i}(x,w_{1},w_{2},\nabla w_{1},\nabla w_{2}) \\ 
\leq M_{i}\tilde{c}_{1}C^{|\alpha _{i}|+|\beta _{i}|}d(x)^{\alpha _{i}+\beta
_{i}}+2C^{\max \{\gamma _{i},\theta _{i}\}}\text{ in }\Omega ,\text{ for }%
i=1,2.%
\end{array}
\label{13}
\end{equation}%
Then, since $\alpha _{i}+\beta _{i}>-1,$ Hardy-Sobolev inequality (see,
e.g., \cite[Lemma 2.3]{AC}) guarantees that $f_{i}(x,w_{1},w_{2},\nabla
w_{1},\nabla w_{2})\in W^{-1,p_{i}^{\prime }}(\Omega )$, for $i=1,2$, while
the unique solvability of $(u,v)$ in $(\mathrm{P}_{w})$ is readily derived
from Minty Browder's Theorem (see, e.g., \cite{Brezis}). Thus, $\mathcal{T}$
is well defined. Moreover, the regularity theory up to the boundary in \cite%
{L} or \cite{Hai}, depending on whether the sign of $\alpha _{i}+\beta _{i}$
is positive ($\alpha _{i}+\beta _{i}\geq 0$) or negative ($\alpha _{i}+\beta
_{i}\in (-1,0)$), respectively, yields $(u,v)\in C_{0}^{1,\sigma }(\overline{%
\Omega })\times C_{0}^{1,\sigma }(\overline{\Omega })$ and there is a
constant $R>0$ and $\sigma \in (0,1)$ such that it holds%
\begin{equation}
\left\Vert u\right\Vert _{C^{1,\sigma }(\overline{\Omega })},\left\Vert
v\right\Vert _{C^{1,\sigma }(\overline{\Omega })}<R\text{.}  \label{14}
\end{equation}

It is worth noting that the weak solutions of problem $(\mathrm{P}_{w})$
coincide with the fixed points of the operator $\mathcal{T}$. To reach the
desired conclusion, we shall apply Schauder's fixed point theorem (see, for
example, \cite[p. 57]{Z}).

\mathstrut

We claim that the map $\mathcal{T}:\mathcal{O}_{C}\rightarrow C_{0}^{1}(%
\overline{\Omega })\times C_{0}^{1}(\overline{\Omega })$ is continuous and
compact. Let $(w_{1,n},w_{2,n})\in \mathcal{O}_{C}$ such that $%
(w_{1,n},w_{2,n})\longrightarrow (w_{1},w_{2})$ in $C_{0}^{1,\sigma }(%
\overline{\Omega })\times {C_{0}^{1,\sigma }(\overline{\Omega })}.$ Setting $%
\mathcal{T}(w_{1,n},w_{2,n})=(u_{n},v_{n})$ reads as 
\begin{equation}
\left\{ 
\begin{array}{l}
\left\langle -\Delta _{p_{1}}u_{n},\varphi \right\rangle =\int_{\Omega
}f_{1}(w_{1,n},w_{2,n}\nabla w_{1,n},\nabla w_{2,n})\varphi \text{ }\mathrm{d%
}x, \\ 
\left\langle -\Delta _{p_{2}}v_{n},\psi \right\rangle =\int_{\Omega
}f_{2}(w_{1,n},w_{2,n}\nabla w_{1,n},\nabla w_{2,n})\psi \text{ }\mathrm{d}x,%
\end{array}%
\right.  \label{88}
\end{equation}%
for all $(\varphi ,\psi )\in W_{0}^{1,p_{1}}(\Omega )\times
W_{0}^{1,p_{2}}(\Omega ).$ Acting in (\ref{88}) with $(\varphi ,\psi
)=(u_{n}-u,v_{n}-v)$, (\ref{13}) and (\ref{11}) allow to apply Lebesgue's
dominated convergence theorem to deduce that 
\begin{equation*}
\lim_{n\longrightarrow \infty }\left\langle -\Delta
_{p_{1}}u_{n},u_{n}-u\right\rangle =\lim_{n\rightarrow \infty }\left\langle
-\Delta _{p_{2}}v_{n},v_{n}-v\right\rangle =0.
\end{equation*}%
Here the integrability of the functions $d(x)^{\alpha _{1}+\beta
_{1}}(u_{n}-u)$ and $d(x)^{\alpha _{2}+\beta _{2}}(v_{n}-v),$ which is known
by invoking the Hardy-Sobolev inequality under assumption (\ref{11}), is
essential. At this point, the $S_{+}$-property of $-\Delta _{p_{i}}$ on $%
W_{0}^{1,p_{i}}(\Omega )$ ensures that 
\begin{equation}
(u_{n},v_{n})\longrightarrow (u,v)\text{ in }W_{0}^{1,p_{1}}(\Omega )\times
W_{0}^{1,p_{2}}(\Omega ).  \label{15}
\end{equation}%
Hence, passing to the limit in (\ref{88}) leads to $(u,v)=\mathcal{T}%
(w_{1},w_{2})$. On the other hand, from (\ref{14}) we know that the sequence 
$\{(u_{n},v_{n})\}$ is bounded in $C_{0}^{1,\sigma }(\overline{\Omega }%
)\times {C_{0}^{1,\sigma }(\overline{\Omega })}$ for certain $\sigma \in
(0,1)$. Since the embedding $C_{0}^{1,\sigma }(\overline{\Omega })\subset
C_{0}^{1}(\overline{\Omega })$ is compact and taking into account (\ref{15}%
), it turns out that along a relabeled subsequence one has that $%
(u_{n},v_{n})\longrightarrow (u,v)$ in $C_{0}^{1}(\overline{\Omega })\times
C_{0}^{1}(\overline{\Omega }).$ Therefore, $\mathcal{T}$ is continuous.
Again, through (\ref{14}), it follows that $\mathcal{T}(\mathcal{O}_{C})$ is
bounded in $C_{0}^{1,\sigma }(\overline{\Omega })\times {C_{0}^{1,\sigma }(%
\overline{\Omega })}.$ Then, invoking the compactness of the embedding $%
C_{0}^{1,\sigma }(\overline{\Omega })\subset C_{0}^{1}(\overline{\Omega })$,
we conclude that $\mathcal{T}(\mathcal{O}_{C})$ is a relatively compact
subset of $C_{0}^{1}(\overline{\Omega })\times C_{0}^{1}(\overline{\Omega })$%
. This proves the claim.

We are left to show that $\mathcal{O}_{C}$ is invariant under $\mathcal{T}$.
Let $(w_{1},w_{2})\in \mathcal{O}_{C}$ and denote $(u,v)=\mathcal{T}%
(w_{1},w_{2})$. Using $(\mathrm{P}_{w})$ and combining Proposition \ref{P1}
with (\ref{17}), the weak comparison principle yields%
\begin{equation}
\begin{array}{l}
\underline{u}\leq u\leq \overline{u}\text{ \ and \ }\underline{v}\leq v\leq 
\overline{v}\text{ in }\Omega .%
\end{array}
\label{23}
\end{equation}%
On the other hand, on the basis of (\ref{13}) and (\ref{cdt}), Theorem \ref%
{T4} entails%
\begin{equation}
\Vert \nabla u\Vert _{\infty },\Vert \nabla v\Vert _{\infty }\leq
k_{p}\max_{i=1,2}||f_{i}(\cdot ,u,v,\nabla u,\nabla v)\Vert _{r_{i}}^{\frac{1%
}{p_{i}-1}},  \label{21}
\end{equation}%
for some positive constant $k_{p}$ depending only on $p_{i}$ and $\Omega $.
Exploiting once again (\ref{13}) one has%
\begin{equation}
\begin{array}{l}
||f_{i}(\cdot ,w_{1},w_{2},\nabla w_{1},\nabla w_{2})\Vert _{r_{i}} \\ 
\leq M_{i}\tilde{c}_{1}C^{|\alpha _{i}|+|\beta _{i}|}\left( \int_{\Omega
}d(x)^{r_{i}(\alpha _{i}+\beta _{i})}\text{ }\mathrm{d}x\right)
^{1/r_{i}}+2|\Omega |C^{\max \{\gamma _{i},\theta _{i}\}}.%
\end{array}
\label{19}
\end{equation}%
Bearing in mind (\ref{cdt}), \cite[Lemma]{LM} applies, showing that the
integral term $\left( \int_{\Omega }d(x)^{r_{i}(\alpha _{i}+\beta _{i})}%
\text{ }\mathrm{d}x\right) ^{1/r_{i}}$ is bounded by some constant $\mu >0$.
At this point, from (\ref{19}) and (\ref{cdt}), we have%
\begin{equation}
\begin{array}{l}
||f_{i}(\cdot ,w_{1},w_{2},\nabla w_{1},\nabla w_{2})\Vert _{r_{i}} \\ 
\leq M_{i}\tilde{c}_{1}C^{|\alpha _{i}|+|\beta _{i}|}\mu +2|\Omega |C^{\max
\{\gamma _{i},\theta _{i}\}}\leq (k_{p}^{-1}C)^{p_{i}-1},%
\end{array}%
\end{equation}%
provided $C>1$ large enough. Hence, owing to (\ref{21}) we conclude that%
\begin{equation}
\Vert \nabla u\Vert _{\infty },\Vert \nabla v\Vert _{\infty }\leq C.
\label{22}
\end{equation}%
Gathering (\ref{23}) and (\ref{22}) enable us to infer that $(u,v)\in 
\mathcal{O}_{C}$, thereby the inclusion $\mathcal{T}(\mathcal{O}_{C})\subset 
\mathcal{O}_{C}$ holds true.

\mathstrut

We are thus in a position to apply Schauder fixed point theorem to the map $%
\mathcal{T}:\mathcal{O}_{C}\rightarrow \mathcal{O}_{C}$, which establishes
the existence of $(u,v)\in \mathcal{O}_{C}$ satisfying $(u,v)=\mathcal{T}%
(u,v)$. In view of the definition of $\mathcal{T}$, it turns out that $(u,v)$
is a (positive) smooth solution of problem $($\textrm{P}$)$. This ends the
proof of Theorem \ref{T1}.

\begin{acknowledgement}
The authors were supported by the Directorate-General of Scientific Research
and Technological Development (DGRSDT).
\end{acknowledgement}

\end{document}